\documentclass[a4paper]{amsart}
\usepackage{amsmath,amsthm}
\usepackage{amsfonts}
\usepackage{latexsym}
\usepackage{amsxtra}

\numberwithin{equation}{section}

\newtheorem{theorem}{Theorem}[section]
\newtheorem{corollary}[theorem]{Corollary}
\newtheorem{lemma}[theorem]{Lemma}
\newtheorem{example}[theorem]{Example}
\newtheorem{proposition}[theorem]{Proposition}
\newtheorem{remark}[theorem]{Remark}

\DeclareMathOperator{\idx}{\mathrm{i}}
\DeclareMathOperator{\Ker}{\mathrm{ker}}
\DeclareMathOperator{\Fr}{\mathrm{Fr}\,}
\newcommand{\cl}[1]{\overline{#1}}
\newcommand{\R}{\mathbb{R}}
\newcommand{\N}{\mathbb{N}}
\newcommand{\F}{\mathcal{F}}
 \newcommand{\sign}{\mathop\mathrm{sign}\nolimits}
\newcommand{\e}{\varepsilon}
\newcommand{\s}{\mathfrak{s}}

\title[A note on topological methods for a class of DAEs]{A note on topological methods 
for a class of Differential-Algebraic Equations}
\author{Marco Spadini}
\address[Marco Spadini]{Dipartimento di Matematica Applicata ``G.\ Sansone'', Via S.\ Marta 3, 
50139 Firenze, Italy}
\email{marco.spadini@{}math.unifi.it}

\begin{document}
\begin{abstract}
 We study a particular class of autonomous Differential-Algebraic Equations that are equivalent
to Ordinary Differential Equations on manifolds. Under appropriate assumptions we determine
a straightforward formula for the computation of the degree of the associated
tangent vector field that does not require any explicit knowledge of the manifold. We use 
this formula to study the set of harmonic solutions to periodic perturbations of our
equations. Two different classes of applications are provided.
\end{abstract}

\maketitle

\section{Introduction}
In this paper we  apply topological methods to the study of the set of periodic 
solutions of periodic perturbations of a particular class of differential-algebraic equations
(DAEs). Namely, we consider the following DAE in semi-explicit form:
\begin{equation}\label{NOpert}
 \left\{
\begin{array}{l}
 \dot x=f(x,y),\\
 g(x,y)=0,
\end{array}
\right.
\end{equation}
where $g:U\to\R^s$ and $f:U\to\R^k$ are continuous maps defined on an open connected set
$U\subseteq\R^k\times\R^s$, with $g\in C^\infty$ and $\partial_2g(p,q)$, the partial derivative 
of $g$ with respect to the second variable, invertible for each $(p,q)\in U$. Given $T>0$, 
we will consider $T$-periodic perturbations of $f$ in \eqref{NOpert} and study the set of 
$T$-periodic solution of the resulting $T$-periodic DAE.  Namely, for $\lambda\geq 0$,
we look at the $T$-periodic solutions of 
\begin{equation}\label{Tpert}
 \left\{
\begin{array}{l}
 \dot x=f(x,y)+\lambda h(t,x,y),\\
 g(x,y)=0,
\end{array}
\right.
\end{equation}
where $h:\R\times U\to\R^k$ is continuous and $T$-periodic in the first variable. 
Roughly speaking, we will give conditions ensuring the existence of a connected component 
of elements $(\lambda;x,y)$, $\lambda\geq 0$ and $(x,y)$ a $T$-periodic solution to 
\eqref{Tpert}, that emanates from the set of constant solutions of \eqref{NOpert} and is 
not compact. This kind of results is useful to study existence and multiplicity of 
$T$-periodic solutions of \eqref{Tpert}.

Since $\partial_2g(p,q)$ is invertible for all $(p,q)\in U$, equations 
\eqref{NOpert} and \eqref{Tpert} are index 1 differential algebraic equation and have 
strangeness index $0$ (see e.g.\ \cite{KM}). However, our argument will not require any 
knowledge of the theory of DAEs.  

The assumption on $\partial_2g(p,q)$ implies that $0\in\R^s$ is a regular value of $g$, 
thus $M:=g^{-1}(0)$ is a $C^\infty$ submanifold of $\R^k\times\R^s$. Notice that $M$,
locally, can be represented as graph of some map from an open subset of $\R^k$ to $\R^s$. 
Thus equations
\eqref{NOpert} and \eqref{Tpert} can be locally decoupled. However, globally, this might 
not be true. Observe also that even when $M$ is a graph of some map $\varphi$, it might
happen that the expression of $\varphi$ is complicated (or even impossible to determine
analytically), so that the decoupled version of \eqref{NOpert} or \eqref{Tpert} may be
impractical. We have a very simple example of this fact if we take $k=s=1$, $U=\R\times\R$,
$g(p,q)=q^7+q-p^2$ and $f(p,q)=q$. 

It is well known (compare \cite[\S 4.5]{KM}) and easy to see that, when $\partial_2g(p,q)$ 
is invertible for all $(p,q)\in U$, equation \eqref{NOpert} induces a tangent vector field 
$\Psi$ on $M$, that is, it gives rise to an autonomous ordinary differential equation on 
$M$. Equation \eqref{Tpert}, then, leads to a $T$-periodic perturbation of this ODE.
In our main result (Theorem \ref{rami} below), in order to get information about the set of 
$T$-periodic solutions of \eqref{Tpert}, we apply an argument of  \cite{FS98} about periodic 
perturbation of an autonomous ordinary differential equation on a differentiable manifold. 
The results of \cite{FS98}, however, require some knowledge of the degree of the perturbed 
tangent vector field. In the present setting this means the degree of the tangent vector field 
$\Psi$ on $M$.  Since $M$ is known only implicitly, and the form of $\Psi$ may not be very 
simple, a direct application of \cite{FS98} is of limited interest. Thus, our first step will 
be to determine a formula (Theorem \ref{formuladeg} below) that allows the computation of the 
absolute value of the degree of $\Psi$ by means of the degree of the ``morally'' simpler 
vector field $F:U\to\R^k\times\R^s$, given by
\begin{equation}\label{defF}
(p,q)\mapsto \big(f(p,q),g(p,q)\big).
\end{equation}
We stress the fact (as we shall briefly discuss below) that, since in Euclidean spaces 
vector fields can be regarded as maps and vice versa,  the degree of the vector field $F$ is 
essentially the well known Brouwer degree, with respect to $0$, of $F$ seen as a map. Hence 
the degree of $F$ has a simpler nature than that of $\Psi$ and, as a consequence, it is usually 
easier to compute. 

\smallskip
\noindent\textbf{Notation.} Throughout this paper, $|\cdot|$ will denote the absolute value 
in $\R$ while $|\cdot|_n$ will be the norm in $\R^n$ given by
\[
|a|_n=\sum_{i=1}^n|a_i|\quad\text{for all}\; a=(a_1,\ldots,a_n)\in\R^n,
\]
Thus, coherently with this notation we have $\big|(p,q)\big|_{k+s}=|p|_k+|q|_s$, for 
$(p,q)\in\R^k\times\R^s$.

\section{Associated vector fields}\label{assocVF}
In this section, we associate ordinary differential equations on the manifold $M=g^{-1}(0)$ 
to \eqref{NOpert} and to \eqref{Tpert}, in quite a natural way (compare \cite[\S 4.5]{KM}). 

Let $I\subseteq\R$ be an interval and $W\subseteq\R^n$ be open. Given $r\in\N\cup\{0\}$,
the set of all $W$-valued $C^r$ functions defined on $I$ is denoted by by $C^r(I,W)$. For
simplicity, we use $C(I,W)$ as a synonym of $C^0(I,W)$.

Let $U\subseteq\R^k\times\R^s$ be open and connected, and let $g:U\to\R^s$, $f:U\to\R^k$ and 
$h:\R\times U\to\R^k$ be continuous maps with $g\in C^\infty$ and $\partial_2g(p,q)$ invertible 
for all $(p,q)\in U$. We also assume, throughout this paper, that $h$ is $T$-periodic in the 
first variable for some given $T>0$. 

A solution of \eqref{Tpert} for a given $\lambda\geq 0$ consists of a pair of functions 
$x\in C^1(I,\R^k)$ and $y\in C(I,\R^s)$, $I$  an interval, with the property that 
\[
\left\{
\begin{array}{l}
 \dot x(t)=f\big(x(t),y(t)\big)+\lambda h\big(t,x(t),y(t)\big),\\
 g\big(x(t),y(t)\big)=0,
\end{array}
\right.
\]
for each $t\in I$.
Notice that the assumptions on $g$ and the Implicit Function Theorem imply that $y$ is 
actually a $C^1$ function. In fact, in what follows, it will be convenient to consider 
a solution of \eqref{Tpert} as a function $\zeta:=(x,y)\in C^1(I,\R^k\times\R^s)$. 

Let $(x,y)\in C^1(I,\R^k\times\R^s)$ be a solution of \eqref{Tpert} for a given $\lambda\geq 0$, 
defined on some interval $I\subseteq\R$. Then, differentiating the identity $g\big(x(t),y(t)\big)=0$, 
we get
\[
 \partial_1g\big(x(t),y(t)\big)\dot x(t)+\partial_2g\big(x(t),y(t)\big)\dot y(t)=0,
\]
which yields 
\begin{equation}\label{soly}
\begin{split}
\dot y(t)=-\big[\partial_2g\big(x(t),y(t))\big]^{-1}&\partial_1g\big(x(t),y(t)\big)\\
&\quad\Big[f\big(x(t),y(t)\big)+\lambda h\big(t,x(t),y(t)\big)\Big]
\end{split}
\end{equation}
for all $t\in I$.

As already observed, because of the assumptions on $\partial_2g(p,q)$, $0\in\R^s$ is a 
regular value of $g$. Thus, $M=g^{-1}(0)$ is a $C^\infty$ submanifold of $\R^k\times\R^s$ 
and, given $(p,q)\in M$, the tangent space $T_{(p,q)}M$ to $M$ at $(p,q)$ is given by the 
kernel $\Ker d_{(p,q)}g$ of the differential  $d_{(p,q)}g$ of $g$ at  $(p,q)$.

Consider $\Psi:M\to\R^k\times\R^s$ and $\Upsilon:\R\times M\to\R^k\times\R^s$ given by
\begin{subequations}\label{campiv}
\begin{align}
\Psi(p,q) &= 
      \left(f(p,q)\,,\,-[\partial_2g(p,q)]^{-1}\partial_1g(p,q)f(p,q)\right),\label{campiv1}\\
\intertext{and}
 \Upsilon(t,p,q)
  &=\left(h(t,p,q),
                -[\partial_2g(p,q)]^{-1}\partial_1g(p,q)h(t,p,q)\right).\label{campiv2}
\end{align}
\end{subequations}
Clearly, $\Upsilon$ is $T$-periodic in the first variable. Let us show that $\Psi$ and 
$\Upsilon$ are tangent to $M$ in the sense that, for any $(t,p,q)\in\R\times M$,
\begin{equation}\label{tang}
 \Psi(p,q)\in T_{(p,q)}M\quad\text{and}\quad\Upsilon(t,p,q)\in T_{(p,q)}M.
\end{equation}
Consider for instance $\Psi$. We have
\[
d_{(p,q)}g [\Psi(p,q)]=
   \begin{pmatrix}
   \partial_1 g(p,q) & \partial_2g(p,q)
   \end{pmatrix}
   \begin{pmatrix}
   f(p,q) \\-[\partial_2g(p,q)]^{-1}\partial_1g(p,q)f(p,q)
   \end{pmatrix}
=0.
\]
Since $T_{(p,q)}M=\Ker d_{(p,q)}g$, the first relation in \eqref{tang} is proved. The second 
one follows from a similar argument and is left to the reader. 

 Taking \eqref{soly} into account, one can see that \eqref{Tpert} is equivalent to the 
following ODE on $M$:
\begin{equation}\label{TeqonM}
\dot\zeta=\Psi(\zeta)+\lambda\Upsilon(t,\zeta),\qquad \lambda\geq 0,
\end{equation}
where, we recall, $\zeta=(x,y)$. By the same argument one can see that \eqref{NOpert} is 
equivalent to 
\begin{equation}\label{NOponM}
\dot\zeta=\Psi(\zeta).
\end{equation}
\begin{remark}
Let $g$ and $f$ and $h$ be as above. When $f$ is $C^1$, so is the vector field $\Psi$. 
Thus, by virtue of the equivalence of \eqref{NOpert} with \eqref{NOponM}, the local results 
on existence, uniqueness and continuous dependence of local solutions of the initial value 
problems translate to \eqref{NOpert} from the theory of ordinary differential equations on 
manifolds. Of course, if also $h$ is $C^1$, a similar statement holds for \eqref{Tpert}.
\end{remark}

Notice that the importance of the hypotheses on $g$ goes beyond ensuring the smoothness
of $M$. In fact, even when $M$ is a differentiable manifold and $g$ is $C^\infty$, if we 
drop our assumption on $\partial_2g$, \eqref{NOpert} may fail to induce a (continuous)
tangent vector field $\Psi$ on $M$ and, even if this happens, \eqref{NOpert}  might not be
equivalent to \eqref{NOponM}. The following simple examples illustrates these possibilities.
\begin{example}\label{EX.1}
Take $k=s=1$ and let $U=\R\times\R$. Consider the following DAE: 
\begin{equation}\label{eqex.1}
 \dot x=1,\qquad x-y^3=0.
\end{equation}
Clearly, $M=\{(p,q)\in\R\times\R:p=q^3\}$ is a $C^\infty$ submanifold of $\R\times\R$. 
Equation \eqref{eqex.1} induces the vector field 
\[
(p,q)\mapsto\left(1,\frac{1}{3q^2}\right)
\]
on all points of $M$, with the exception of $(0,0)$. Clearly, this vector field cannot
be extended to a continuous tangent vector field on $M$.
\end{example}
\begin{example}\label{EX.2}
Let $k=s=1$ and let $U=\R\times\R$, as in the previous example. Consider the following DAE:
\begin{equation}\label{eqex.2}
 \dot x=y,\qquad x^2+y^2=1.
\end{equation}
In this case, the manifold $M$ is the unit circle $S^1$ of $\R\times\R$ centered at the 
origin. Clearly, \eqref{eqex.2} induces on $S^1\setminus\{(\pm 1,0)\}$ the vector field 
$(p,q)\mapsto (q,-p)$ that can be extended uniquely to a vector field $\Psi$ defined on the 
whole $S^1$. Notice, however, that \eqref{eqex.2} is not equivalent to \eqref{NOponM} on 
$S^1$. In fact, the maps $t\mapsto(\pm 1,0)$ are solutions of \eqref{eqex.2}, but not of 
\eqref{NOponM}.
\end{example}

Observe that taking $U=(\R\times\R)\setminus\{(0,0)\}$ in Example \ref{EX.1}, the manifold 
$M=\{(p,q)\in U:p=q^3\}$ consists of two connected sets which the vector field 
$\Psi(p,q)=\big(1,1/(3q^2)\big)$ is tangent to. Now, \eqref{eqex.1} turns out 
to be equivalent to \eqref{NOponM} on $M$. 

Similarly, taking $U=(\R\times\R)\setminus\{(\pm 1,0)\}$ in Example \ref{EX.2}, one has that 
$M$ consists of two connected components and the above construction of $\Psi$ can be carried 
out on $M$.

\medskip
In order to investigate the $T$-periodic solutions of \eqref{Tpert} we will study the 
set of $T$-periodic solutions of the equivalent equation \eqref{TeqonM}. Our first 
step will be to consider the case $\lambda=0$ and determine a formula for the computation
of the degree (sometimes called characteristic or rotation) of the tangent vector field
$\Psi$ on $U$. Before doing that, however, we will recall some basic facts about the
notion of the degree of a tangent vector field.

\section{The degree of a tangent vector field}

\smallskip
We now recall some basic notions about tangent vector fields on manifolds.

Let $M\subseteq\R^n$ be a manifold. Given any $p\in M$, $T_pM\subseteq\R^n$ denotes the 
tangent space of M at $p$. Let $w$ be a tangent vector field on $M$, that is, a continuous 
map $w:M\to\R^n$ with the property that $w(p)\in T_pM$ for any $p\in M$. If $w$ is (Fr\'echet) differentiable at $p\in M$ and $w(p)=0$, then the differential $d_pw : T_pM\to\R^k$ maps 
$T_pM$ into itself (see e.g.\ \cite{Mi}), so that the determinant $\det d_pw$ of $d_pw$ is 
defined. If, in addition, $p$ is a nondegenerate zero (i.e.\ $d_pw : T_pM\to\R^n$ is injective) 
then $p$ is an isolated zero and $\det d_pw\neq 0$.

Let $W$ be an open subset of $M$ in which we assume $w$ admissible for the degree; that is, 
the set $w^{-1}(0) \cap W$ is compact. Then, one can associate to the pair $(w,W)$ an integer, 
$\deg(w,W)$, called the \emph{degree (or characteristic) of the vector field $w$ in $W$}, 
which, roughly speaking, counts (algebraically) the zeros of $w$ in $W$ (see e.g.\ 
\cite{FPS05, H, Mi} and references therein). For instance, when the zeros of $w$ are all 
nondegenerate, then the set $w^{-1}(0)\cap W$ is finite and
\begin{equation}\label{sommasegni}
\deg(w,W)=\sum_{q\in w^{-1}(0)\cap W}{\rm sign} \det d_qw.
\end{equation}
When $M = \R^n$, $\deg(w,W)$ is just the classical Brouwer degree, $\deg_B(w,V,0)$, where 
$V$ is any bounded open neighborhood of $w^{-1}(0) \cap W$ whose closure is contained in $W$.

For the purpose of future reference, we mention a few of the properties of the degree 
of a tangent vector field that shall be useful in the sequel. Here $W$ is an open subset 
of a manifold $M\subseteq\R^n$ and $w:M\to\R^n$ is a tangent vector field.
\begin{description}
\item[Solution]
{\em If $(w,W)$ is admissible and $\deg(w,W)\neq 0$, then $w$ has a zero in $W$.}
\item[Additivity]
{\em Let $(w,W)$ be admissible. If $W_1$ and $W_2$ are two
disjoint open subsets of $W$ whose union contains $w^{-1}(0)\cap W$, then}
\[
\deg(w,W) = \deg(w,W_1)+\deg(w,W_2).
\]
\item[Homotopy Invariance]
{\em Let $h:M\times[0,1]\to\R^n$ be a homotopy of tangent vector fields admissible in 
$W$; that is, $h(p,\lambda)\in T_pM$ for all $(p,\lambda)\in M\times[0,1]$ and 
$h^{-1}(0)\cap W\times[0,1]$ is compact. 
Then $\deg\big(h(\cdot,\lambda),W\big)$ is independent of $\lambda$.}
\item[Invariance under diffeomorphisms] {\em Let $M\subseteq\R^m$ and $N\subseteq\R^n$ 
be differentiable manifolds and let $v:N\to\R^n$ and $w:M\to\R^m$ be tangent vector fields.
Let also $V\subseteq N$ and $W\subseteq M$ be open, and let $\varphi:W\to V$ be a 
diffeomorphism. If 
\[
v(q)=d_{\varphi^{-1}(q)}\varphi\big[w\big(\varphi^{-1}(q)\big)\big]\quad
\forall q\in V,
\]
we say that $v|_V$ and $w|_W$ correspond under the diffeomorphism $\varphi$. In this case, 
if either $v$ is admissible in $V$ or or $w$ is admissible in $W$, then so is the other and 
\[
\deg(v,V)=\deg(w,W).
\]}
\end{description}
\begin{remark}\label{sper}
Let $M\subseteq\R^n$ be a differentiable manifold and let $W\subseteq M$ be open and 
relatively compact. If $w:M\to\R^n$ is such that $w(p)\neq 0$ on the boundary $\Fr(W)$ 
of $W$, then $(w,W)$ is admissible. Let $\e=\min_{p\in\Fr(W)}|w(p)|_n$. Then, for any 
$v:M\to\R^n$ such that $\max_{p\in\Fr(W)}|w(p)-v(p)|_n<\e$, we have that $(v,W)$ is 
admissible and that the homotopy $h:M\times[0,1]\to\R^n$ given by 
\[
h(p,\lambda)=\lambda w(p)+(1-\lambda) v(p)
\]
is admissible in $W$. Hence, by the Homotopy Invariance Property, 
\[
\deg(w,W)=\deg(v,W).
\]
\end{remark}

The Additivity Property implies the following important one:
\begin{description}
 \item[Excision]{\em Let $(w,W)$ be admissible. If $V\subseteq W$ is open and contains 
$w^{-1}(0)\cap V$, then $\deg(w,W) = \deg(w,V)$.} 
\end{description}

The Excision Property allows the introduction of the notion of index of an isolated zero 
of a tangent vector field. Let $w:M\to\R^n$ be a vector field tangent to the differentiable 
manifold $M\subseteq\R^n$, and let $q\in M$ be an isolated zero of $w$. Clearly, $\deg(w,V)$ 
is well defined for each open $V\subseteq M$ such that $V\cap w^{-1}(0)=\{ q\}$. By the 
Excision Property $\deg(w,V)$ is constant with respect to such $V$'s. This common value of 
$\deg(w,V)$ is, by definition, the \emph{index of $w$ at $ q$}, and is denoted by 
$\mathrm{i}\,(w, q)$. Using this notation, if $(w,W)$ is admissible, by the Additivity 
Property we get that if all the zeros in $W$ of $w$ are isolated, then
\begin{equation}\label{sommaindici}
\deg(w,W)=\sum_{q\in w^{-1}(0)\cap W} \mathrm{i}\,(w,q).
\end{equation}
By formula \eqref{sommasegni} we have that if $q$ is a nondegenerate zero of $w$, then
\[
 \mathrm{i}\,(w,q)=\sign\det d_ q w.
\]
Notice that \eqref{sommasegni} and \eqref{sommaindici} differ in the fact that, in the
latter, the zeros of $w$ are not necessarily nondegenerate as they have to be in the former. 
In fact, in \eqref{sommaindici}, $w$ need not be differentiable at its zeros.

\section{The degree of $\Psi$}

In this section we shall obtain a simple formula for the computation of the degree of $\Psi$ 
that does not require an ``explicit'' expression of the manifold $M=g^{-1}(0)$.  Namely,
let $U\subseteq\R^k\times\R^s$ be be open and connected. Define $F:U\to\R^k\times\R^s$ 
by \eqref{defF} and $\Psi:M\to\R^k\times\R^s$ by \eqref{campiv1}. We shall prove 
a formula that allows the computation of $\deg(\Psi,M)$ from the degree of $F$ in $U$ 
(notice that $U\subseteq\R^k\times\R^s$ being an open set is a differentiable manifold, so 
that $\deg(F,U)$ makes sense).

\begin{theorem}\label{formuladeg}
Let $U\subseteq\R^k\times\R^s$ be open and connected, and let $g:U\to\R^s$ and $f:U\to\R^k$ 
be such that $f$ is continuous and $g$ is $C^\infty$ with $\partial_2g(p,q)$ invertible 
for all $(p,q)\in U$. Let also $F:U\to\R^k\times\R^s$ and $\Psi:M\to\R^k\times\R^s$ be given 
by \eqref{defF} and \eqref{campiv1}, respectively. If either $\deg(\Psi,M)$ or $\deg(F,U)$ 
is well defined, so is the other, and
\begin{equation}\label{idgradi}
 \big|\deg(\Psi,M)\big|=\big|\deg(F,U)\big|.
\end{equation}
\end{theorem}

The proof of Theorem \ref{formuladeg} makes use of the following technical lemma. 

\begin{lemma}\label{approxf}
Let $U$, $f$ and $g$ be as in Theorem \ref{formuladeg}. Given $\e>0$, there exists a $C^1$ 
map $f_\e:U\to\R^k$ such that 
\[
 \sup_{(p,q)\in U}\big|f_\e(p,q)-f(p,q)\big|_k<\e
\]
and such that $(0,0)\in\R^k\times\R^s$ is a regular value of the map $F_\e:U\to\R^k\times\R^s$
given by $F_\e(p,q)=\big(f_\e(p,q),g(p,q)\big)$.
\end{lemma}
\begin{proof}
What follows is a fairly usual argument in transversality theory (see e.g.\ \cite{GP}). For the 
sake of completeness, though, we will provide a complete proof.

By standard approximation results in Euclidean spaces, there exists a $C^1$ map 
$\tilde f:U\to\R^k$ such that
\[
 \sup_{(p,q)\in U}\big|\tilde f(p,q)-f(p,q)\big|_k<\frac{\e}{2}.
\]

Denote by $B$ the $\frac{\e}{2}$-ball of $\R^k$ centered at the origin, and define 
$\F:\R^k\times\R^s\times B\to\R^k\times\R^s$ by $\F(p,q,b)=\big(\tilde f(p,q)+b,g(p,q)\big)$.
Since the origin of $\R^s$ is a regular value for $g$, the origin $(0,0)\in\R^k\times\R^s$ is 
a regular value for $\F$. Thus, $X=\F^{-1}(0,0)$ is a $k$-dimensional $C^1$ submanifold of
$\R^k\times\R^s\times\R^k$. 

Denote by $\pi$ the projection of $\R^k\times\R^s\times B$ onto its third factor. Clearly,
the restriction $\pi|_X$ of $\pi$ to $X$ is $C^1$. By the well known Morse--Sard Theorem 
(see e.g.\ \cite{H}) it is possible to choose an element $\bar b\in B$ which is a regular value 
for $\pi|_X$. Let us show that, with such a choice of $\bar b$, $(0,0)\in\R^k\times\R^s$ is a 
regular value of the map $\F_{\bar b}=\F(\cdot,\cdot,\bar b)$.

To see that, we need to show that for any $(p,q)\in\F_{\bar b}^{-1}(0,0)$, the differential
$d_{(p,q)}\F_{\bar b}:T_{(p,q)}(\R^k\times\R^s)\to T_{(0,0)}(\R^k\times\R^s)$ of $\F_{\bar b}$ 
at $(p,q)$ is surjective. We will prove that, given any $\alpha=(\alpha_1,\alpha_2)\in T_{(p,q)}(\R^k\times\R^s)=\R^k\times\R^s$, there exists $v\in T_{(0,0)}(\R^k\times\R^s)=\R^k\times\R^s$ such that $d_{(p,q)}\F_{\bar b} v=\alpha$.

Denote by $d_{(p,q,\bar b)}\F$ the differential  of $\F$ at $(p,q,\bar b)$. Since 
$(0,0)\in\R^k\times\R^s$ is a regular value for $\F$, there exists an element $(w_1,w_2,e)$ 
of $\R^k\times\R^s\times\R^k$ (i.e.\ of the tangent space to $\R^k\times\R^s\times B$ at 
$(p,q,\bar b)$) such that $d_{(p,q,\bar b)}\F(w_1,w_2,e)=\alpha$. Moreover, since $\bar b$ is 
a regular value for $\pi|_X$, the differential of $\pi|_X$ at $(p,q,\bar b)$ 
\[
d_{(p,q,\bar b)}\pi|_X:T_{(p,q,\bar b)}X\to T_{\bar b}B =\R^k
\]
is surjective. Thus, there exists an element $(u_1,u_2,e)\in T_{(p,q,\bar b)}X$ such that
\[
d_{(p,q,\bar b)}\pi|_X(u_1,u_2,e)=e.
\] 
Observe that $d_{(p,q,\bar b)}\F(u_1,u_2,e)=(0,0)$ 
because, as it is well known, $T_{(p,q,\bar b)}X=\ker d_{(p,q,\bar b)}\F$. Thus, taking $v=(w_1-u_1,w_2-u_2)$, and $\bar v=(w_1-u_1,w_2-u_2,0)$
we have
\begin{align*}
d_{(p,q)}\F_{\bar b} v &= d_{(p,q,\bar b)}\F\bar v
                    = d_{(p,q,\bar b)}\F (w_1-u_1,w_2-u_2,0)\\
                    &= d_{(p,q,\bar b)}\F\big( (w_1,w_2,e)-(u_1,u_2,e)\big)\\
                    &= d_{(p,q,\bar b)}\F(w_1,w_2,e)=\alpha.
\end{align*}
Thus, $(0,0)\in\R^k\times\R^s$ is a regular value of $\F_{\bar b}$ as claimed. 

To conclude the proof it is now sufficient to define $f_\e(p,q)=\tilde f(p,q)+\bar b$ for all 
$(p,q)\in U$.
\end{proof}

\begin{proof}[Proof of Theorem \ref{formuladeg}.]

 The first part of the assertion is an obvious consequence of the fact that $F^{-1}(0,0)$ 
coincides with the set $\{(p,q)\in M:\Psi(p,q)=(0,0)\}$. 

We now proceed to prove \eqref{idgradi}. Let $\s$ be the constant sign of $\det\partial_2g(p,q)$ 
in the connected set $U$. The following formula: 
\begin{equation}\label{idgradi_ws}
 \deg(\Psi,M)=\s  \deg(F,U),
\end{equation}
obviously imply \eqref{idgradi}. Let us prove \eqref{idgradi_ws}. Let $V$ be an open and 
bounded subset of $U$ with the property that the closure $\cl{V}$ of $V$ is contained in 
$U$. Assume that $F^{-1}(0,0)\subseteq V$, clearly one has that $\Psi^{-1}(0,0)$ is contained 
in $V$ as well and, by the excision property of the degree of a vector field, we get
\[
 \deg(F,U)=\deg(F,V),\qquad \deg(\Psi, M)=\deg(\Psi,V\cap M).
\]
Therefore it is sufficient to prove that $\deg(\Psi,V\cap M)=\s\deg(F,V)$. 

Let $\e=\min\{|F(p,q)|_{k+s}:(p,q)\in\Fr(V)\}$. By Lemma \ref{approxf}, one can find a 
$C^1$ map $f_\e:U\to\R^k$, with 
\begin{equation}\label{smallper}
  \sup_{(p,q)\in U}\big|f_\e(p,q)-f(p,q)\big|_k<\e,
\end{equation}
and such that $(0,0)$ is a regular value of $F_\e:U\to\R^k\times\R^s$ given by 
$F_\e(p,q)=\big(f_\e(p,q),g(p,q)\big)$. Consider $\Psi_\e:M\to\R^k\times\R^s$
given by 
\[
\Psi_\e(p,q)=
      \left(f_\e(p,q),-[\partial_2g(p,q)]^{-1}\partial_1g(p,q)f_\e(p,q)\right)
\]
for any $(p,q)\in M$. Clearly $\Psi_\e$ is tangent to $M$. By \eqref{smallper} we have 
\[
  \max_{(p,q)\in\Fr(V)}\big|F_\e(p,q)-F(p,q)\big|_{k+s}<\e,
\]
so that, as in Remark \ref{sper}, we have that $\deg(F_\e,V)=\deg(F,V)$. Also, the homotopy 
$(p,q;\lambda)\mapsto \lambda\Psi_\e(p,q)+(1-\lambda)\Psi(p,q)$ is admissible on $V\cap M$
since its $\R^k$-component never vanishes for $(p,q;\lambda)\in\Fr(V\cap M)\times[0,1]$. 
Thus, $\deg(\Psi_\e,V\cap M)=\deg(\Psi,V\cap M)$. Therefore, it is sufficient to show that $\deg(\Psi_\e,V\cap M)=\s\deg(F_\e,V)$.
 
As with $F$ and $\Psi$, one has that $F_\e(p,q)=(0,0)$ if and only if $(p,q)\in M$ and 
$\Psi_\e(p,q)=(0,0)$. Since $(0,0)$ is a regular value of $F_\e$, all the zeros of 
$F_\e$ are nondegenerate, thus isolated. Since $\cl{V}$ is compact, $F_\e^{-1}(0,0)$ 
is finite. 
Let $F_\e^{-1}(0,0)=\{(p_i,q_i)\}_{i=1,\ldots,n}$. Clearly, for each $i=1,\ldots,n$, 
$(p_i,q_i)\in M$ and $(p_i,q_i)$ is an isolated zero of $\Psi_\e$. From \eqref{sommasegni}
and \eqref{sommaindici} we have
\begin{gather*}
 \deg(F_\e,V)=\sum_{i=1}^n\sign\det d_{(p_i,q_i)}F_\e\\
 \deg(\Psi_\e,V\cap M)=\sum_{i=1}^n\idx\big(\Psi_\e,(p_i,q_i)\big)
\end{gather*}
The assertion follows if we prove that
\begin{equation}\label{indiciloc}
 \idx\big(\Psi_\e,(p_i,q_i)\big)
      =\big(\sign\det\partial_2 g(p_i,q_i)\big)\,\big(\sign\det d_{(p_i,q_i)}F_\e\big)
\end{equation}
for $i=1,\ldots,n$.

Let $i\in\{1,\ldots,n\}$ be fixed. In order to compute $\sign\det d_{(p_i,q_i)}F_\e$ we 
write $d_{(p_i,q_i)}F_\e$ in block-matrix form:
\[
 d_{(p_i,q_i)}F_\e = 
\begin{pmatrix}
 \partial_1 f_\e(p_i,q_i) & \partial_2 f_\e(p_i,q_i)\\
\partial_1 g(p_i,q_i) & \partial_2 g(p_i,q_i)
\end{pmatrix}.
\]
Being $\det\partial_2 g(p_i,q_i)\neq 0$, the so-called generalized Gauss algorithm (see 
e.g.\ \cite{G}) yields
\begin{equation}
\begin{split}\label{fsgnsch}
 \det d_{(p_i,q_i)} & F_\e = \det\partial_2 g(p_i,q_i)\cdot\\
     &\cdot\det \Big(\partial_1 f_\e(p_i,q_i)-\partial_2 f_\e(p_i,q_i)
                     \big(\partial_2 g(p_i,q_i)\big)^{-1}\partial_1 g(p_i,q_i)\Big).
\end{split}
\end{equation}

Let $W_i$ be a neighborhood of $(p_i,q_i)$ in $\R^k\times\R^s$ such that $F_\e(p,q)\neq(0,0)$ 
for any $(p,q)\in W_i\setminus\{(p_i,q_i)\}$. Clearly, $\Psi_\e(p,q)\neq (0,0)$ for any 
$(p,q)\in W_i\cap M\setminus\{(p_i,q_i)\}$. Without loss of generality we can assume that
$W_i=U_i\times V_i$ for appropriate open sets $U_i\subseteq\R^k$ and $V_i\subseteq\R^s$.

Since $\partial_2g(p_i,q_i)$ is invertible, the implicit function theorem implies that, taking
a smaller $U_i$ if necessary, we can assume that there exists a $C^1$ function 
$\gamma_i:U_i\to\R^s$ such that $g\big(p,\gamma_i(p)\big)=0$ for any $p\in U_i$. The continuity
of $\gamma_i$ imply that, taking again a smaller $U_i$ if necessary, we can assume 
$\gamma_i(U_i)\subseteq V_i$. Thus the map $G_i:p\mapsto\big(p,\gamma_i(p)\big)$ is a 
diffeomorphism of $U_i$ onto $W_i\cap M$, its inverse being the projection 
$\pi:W_i\cap M\to U_i$ given by $\pi(p,q)=p$.

The property of invariance under diffeomorphisms of the degree of tangent vector fields implies
that
\[
 \deg(\Psi_\e,W_i\cap M)=\deg\big(\pi\circ\Psi_\e\circ G_i,U_i).
\]
Notice that $p_i$ is an isolated zero of $\pi\circ\Psi_\e\circ G_i$. The differential
of this map at $p_i$ is 
\[
 \partial_1 f_\e(p_i,q_i)-\partial_2 f_\e(p_i,q_i)
\big(\partial_2 g(p_i,q_i)\big)^{-1}\partial_1 g(p_i,q_i)
\]
(recall that $q_i=\gamma_i(p_i)$). By \eqref{fsgnsch} and the fact that $(0,0)$ is a regular value 
for $F_\e$, it follows that this differential is invertible. Therefore we have 
\begin{equation}
\begin{split}\label{idxpsi}
\idx\big(\Psi_\e, & (p_i,q_i)\big)=\\
  &=\sign\det \Big(\partial_1 f_\e(p_i,q_i)-\partial_2 f_\e(p_i,q_i)
                     \big(\partial_2 g(p_i,q_i)\big)^{-1}\partial_1 g(p_i,q_i)\Big).
\end{split}
\end{equation}
Equations \eqref{fsgnsch} and \eqref{idxpsi} clearly imply \eqref{indiciloc}. The
assertion follows.
\end{proof}

\begin{remark}\label{fprec}
Let $U$, $f$, $g$, $M$, $\Psi$ and $F$ be as in Theorem \ref{formuladeg}. In fact, an
inspection of its proof reveals that we have proved a slightly more precise, albeit 
less elegant, formula concerning $\deg(\Psi,M)$. Namely,
\begin{equation*}
 \deg(\Psi,M)=\sign\big(\det\partial_2g(p,q)\big)  \deg(F,U).
\end{equation*}
(Recall that $\det\partial_2g(p,q)$ has constant sign for $(p,q)$ in the connected set $U$.)
\end{remark}

Let us illustrate Theorem \ref{formuladeg} with two examples.

\begin{example}
Consider the following second order DAE in $\R\times\R$:
\begin{equation}\label{pozzo}
\left\{
\begin{array}{l}
 \ddot x = -x+y-\dot x,\\
 y^3+y-x^2 =0.
\end{array}\right.
\end{equation}
We rewrite \eqref{pozzo} as the following equivalent first order system in $U=\R^2\times\R$:
\begin{equation}\label{parab}
 \left\{
\begin{array}{l}
 \dot x_1 = x_2,\\
 \dot x_2= -x_1+y-x_2,\\
 y^3+y-x_1^2 =0.
\end{array}\right.
\end{equation}
Let $g:\R^2\times\R\to\R$ be given by $g(p_1,p_2;q)=q^3+q-p_1^2$. As in Section
\ref{assocVF}, Equation \eqref{parab} is equivalent to the ordinary differential equation 
$\dot\zeta=\Psi(\zeta)$ on $M=g^{-1}(0)\subseteq\R^2\times\R$ where 
\[
\Psi(p_1,p_2;q)=\left(p_2, -p_1+q-p_2\,;\,
                  \frac{2p_1p_2}{1+3q^2} \right)
\] 
for all $(p_1,p_2;q)\in M$. Computing $|\deg(\Psi,M)|$ directly from the expression of $\Psi$ 
is possible, of course. However, an easier way is to observe that by Theorem \ref{formuladeg}
we have $\big|\deg(\Psi,M)\big|=\big|\deg(F,U)\big|$, where $F:\R^2\times\R\to\R^2\times\R$ 
is given by 
\[
F(p_1,p_2;q)=\left(p_2, -p_1+q-p_2\,;\,q^3+q-p_1^2\right).
\]
A simple computation shows that the unique zero of $F$ is $(0,0;0)$ and that we have
$\deg(F,U)=1$. Hence, $\big|\deg(\Psi,M)\big|=1$. Actually, according to Remark \ref{fprec}, 
\[
 \deg(\Psi,M)=\sign\big(\det\partial_2 g(p_1,p_2;q)\big)\deg(F,U)=\deg(F,U)=1.
\]
\end{example}

\begin{example} \emph{(Index $2$ DAE in Hessenberg form.)}
Let $U_1$ and $U_2$ be open subsets of $\R^k$, and $\R^s$, respectively; and let 
$U=U_1\times U_2$. Consider the following DAE
\begin{equation}\label{tizio}
 \left\{
\begin{array}{l}
 \dot x= f(x,y),\\
 \gamma(x)=0,
\end{array}\right.
\end{equation}
where $f:U\to\R^k$ and $\gamma:U_1\to\R^s$ are $C^\infty$. Assume that 
$d_p\gamma[\partial_2f(p,q)]:\R^s\to\R^s$ is an invertible linear operator for all
$(p,q)\in U$. In this case \eqref{tizio} is an index $2$ differential-algebraic equation in 
Hessenberg form (see e.g.\ \cite{KM}). With a simple index reduction, we see that 
\eqref{tizio} is equivalent to the following DAE:
\begin{equation}\label{tiziorid}
 \left\{
\begin{array}{l}
 \dot x= f(x,y),\\
 d_x\gamma\big( f(x,y)\big)=0,
\end{array}\right.
\end{equation}
Let us set $g(p,q)=d_p\gamma\big(f(p,q)\big)$ for all $(p,q)\in U$. Then, 
\[
\partial_2g(p,q)=d_p\gamma\big(\partial_2f(p,q)\big):\R^s\to\R^s
\]
is invertible. The vector field $\Psi$, constructed as in Section \ref{assocVF} and 
tangent to $M=g^{-1}(0)$, has the following expression
\begin{multline*}
 \Psi(p,q)=
\Big(f(p,q)\,,\, -\left[d_p\gamma\big(\partial_2f(p,q)\big)\right]^{-1}\\
     \left\{d_p^2\gamma\big(f(p,q),f(p,q)\big)
             +d_p\gamma\big[\partial_1f(p,q)\big]f(p,q)\right\}\Big).
\end{multline*}
where the bilinear form $d_p^2\gamma\big(\cdot,\cdot\big)$ is the second differential of 
$\gamma$ at $p$. Theorem \ref{formuladeg} shows that the rather unappealing task of computing
$|\deg(\Psi,M)|$ reduces to the computation of the comparatively simpler $|\deg(F,U)|$
where,
\[
 F(p,q)=\Big(f(p,q)\,,\,d_p\gamma \big[f(p,q)\big]\Big).
\]
\end{example}

\section{The set of $T$-periodic solutions of \eqref{Tpert}}

This section is devoted to the study of the set of $T$-periodic solutions of equation
\eqref{Tpert}. Recall that $U\subseteq\R^k\times\R^s$ is open and connected, $f:U\to\R^k$, 
$g:U\to\R^s$ and $h:\R\times U\to\R^k$ are continuous, and we assume that $h$ is $T$-periodic 
in the first variable for a given $T>0$, and $g$ is $C^\infty$ with the property that  $\det\partial_2g(p,q)\neq 0$ for all $(p,q)\in U$.

We need to introduce some further notation: denote by $C_T(U)$ the metric subspace of 
the Banach space $C_T(\R^k\times\R^s)$ of all the continuous $T$-periodic functions taking 
values in $U$. We say that $(\mu;x,y)\in [0,\infty)\times C_T(U)$ is a \emph{solution pair 
of} \eqref{Tpert} if $(x,y)$ satisfies \eqref{Tpert} for $\lambda=\mu$; here the pair 
$(x,y)$ is thought of as a single element of $C_T(U)$.  It is convenient, given any 
$(p,q)\in\R^k\times\R^s$, to denote by $(\hat p,\hat q)$ the map in 
$C_T(\R^k\times\R^s)$ that is constantly equal to $(p,q)$. A solution pair of the form
$(0;\hat p,\hat q)$ is called \emph{trivial}. 

As mentioned in the Introduction, the main result of this section, Theorem \ref{rami}
below, follows from a combination of Theorem \ref{formuladeg} and an argument of 
\cite{FS98}, where most of the technical difficulties that arise when working with 
branches of solution pairs are solved (this fact explains the simplicity of the proof
of Theorem \ref{rami}). 

Let $F:U\to\R^k\times\R^s$ be given by 
\eqref{defF}. As one immediately checks, $(\hat p,\hat q)$ is a constant solution of 
\eqref{Tpert} corresponding to $\lambda=0$ if and only if $F(p,q)=(0,0)$. Thus, with this 
notation, the set of trivial solution pairs can be written as 
\[
\{(0;\hat p,\hat q)\in [0,\infty)\times C_T(U): F(p,q)=(0,0)\}.
\] 

Given $\Omega\subseteq[0,\infty)\times C_T(U)$, with $U\cap\Omega$ we denote the set of 
points of $U$ that, regarded as constant functions, lie in $\Omega$. Namely,
\[
  U\cap\Omega=\{(p,q)\in U: (0;\hat p,\hat q)\in\Omega\}.
\]

We are now ready to state and prove our main result concerning the $T$-periodic solutions
of  \eqref{Tpert}. 

\begin{theorem}\label{rami}
Let $U\subseteq\R^k\times\R^s$ be open and connected. Let $g:U\to\R^s$, $f:U\to\R^k$,
$h:\R\times U\to\R^k$ and $T>0$ be such that $f$ and $h$ are continuous, $h$ is 
$T$-periodic in the first variable,
and $g$ is $C^\infty$ with $\partial_2g(p,q)$ invertible for all $(p,q)\in U$. Let also 
$F(p,q)=\big(f(p,q),g(p,q)\big)$. Given $\Omega\subseteq[0,\infty)\times C_T(U)$ open, 
assume $\deg(F,U\cap\Omega)$ is well-defined and nonzero. Then, there exists a 
connected set $\Gamma$ of nontrivial solution pairs of \eqref{Tpert} whose closure in 
$[0,\infty)\times C_T(U)$ meets the set $\{(0,\hat p,\hat q)\in\Omega:F(p,q)=(0,0)\}$ 
and is not contained in any compact subset of $\Omega$.
\end{theorem}

\begin{proof}
Denote by $C_T(M)$  metric subspace of the Banach space $C_T(\R^k\times\R^s)$, of all 
the continuous $T$-periodic functions taking values in $M=g^{-1}(0)$. Let $\Psi$ and 
$\Upsilon$ be as in \eqref{campiv}. Then \eqref{Tpert} is equivalent to \eqref{TeqonM} 
on $M$.

By Theorem \ref{formuladeg} we have that $\deg(\Psi,M\cap\Omega)\neq 0$, here by
$M\cap\Omega$ we mean the set $\{ (p,q)\in M:(0;\hat p,\hat q)\in\Omega\}$. Theorem 3.3 
of \cite{FS98} implies the existence of a connected subset $\Gamma$ of
\[
\big\{ (\lambda;x,y)\in\big([0,\infty)\times C_T(M)\big)\cap\Omega:\text{$(x,y)$ is a 
nonconstant solution of \eqref{TeqonM}}\big\}
\]
whose closure $\cl{\Gamma}$ in $[0,\infty)\times C_T(M)$ is not contained in any compact 
subset of $\big([0,\infty)\times C_T(M)\big)\cap\Omega$ and meets the set 
\[
\big\{(0;\hat p,\hat q)\in\Omega:\Psi(p,q)=(0,0)\big\},
\]
that coincides with $\big\{(0,\hat p,\hat q)\in\Omega:F(p,q)=(0,0)\big\}$.

Clearly, each $(\lambda;x,y)\in\Gamma$ is a nontrivial solution pair of \eqref{Tpert}.
Since $M$ is closed in $U$, it is not difficult to prove that 
$[0,\infty)\times C_T(M)$ is closed in $[0,\infty)\times C_T(U)$. Thus, $\cl\Gamma$ 
coincides with the closure of $\Gamma$ in $[0,\infty)\times C_T(U)$. Consequently 
$\Gamma$ satisfies the assertion.
\end{proof}

\begin{example}
Consider the $2\pi$-periodically perturbed DAE in $U=\R\times (-1,1)$
\begin{equation}\label{equivlien}
 \left\{
\begin{array}{l}
 \dot x = -y- \lambda\sin t,\\
 x-\frac{1}{3}y^3 +y =0,
\end{array}\right.
\end{equation}
that is obtained by writing in the Li\'enard plane the following forced van der Pol 
differential equation and taking the limit as $\e\to 0$:
\[
\e\ddot y+ (y^2-1)\dot y+y+\lambda\sin t=0, \quad y\in (-1,1).
\]
Let $F(p,q)=\big(-q,p-\frac{1}{3}q^3 +q\big)$, and put $\Omega=[0,\infty]\times C_{2\pi}(U)$. 
Clearly, one has $U=U\cap\Omega$, $F^{-1}(0,0)=\{(0,0)\}$ and $\deg(F,U)=1$. Theorem 
\ref{rami}, yields a connected set $\Gamma$ of nontrivial solution pairs of \eqref{equivlien} 
whose closure in $[0,\infty)\times C_T(U)$ meets the trivial solution pair 
$\{(0;\hat 0,\hat 0)\}$ and is not compact. (Here $\hat 0$ denotes the identically zero 
function in $\R$.) 

As one immediately checks, the only $2\pi$-periodic solution of \eqref{equivlien}, for
$-1<y<1$ and $\lambda=0$, is $\{(\hat 0,\hat 0)\}$. Thus, by the connectedness of $\Gamma$ 
one can deduce that \eqref{equivlien} admits a $2\pi$-periodic solution for sufficiently 
small values of $\lambda>0$.
\end{example}

The result of Theorem \ref{rami} is slightly more intuitive when $M=g^{-1}(0)$ is 
closed in $\R^k\times\R^s$ (as, for instance, if $U=\R^k\times\R^s$). In fact, in 
this case, the metric subspace $C_T(M)\subseteq C_T(\R^k\times\R^s)$, that consists 
of all continuous $T$-periodic and $M$-valued functions, is complete. In this situation,
we deduce the following \emph{Continuation Principle} from Theorem \ref{rami}.

\begin{corollary}\label{corami}
Let $f$, $g$, $h$, $U$, $M$, $T$, $F$ and $\Omega$ be as in Theorem \ref{rami}. Assume 
also that $M=g^{-1}(0)$ is closed in $\R^k\times\R^s$. Let $\deg(F,U\cap\Omega)$ be nonzero.
Then there exists a connected component of the set of solution pairs of \eqref{Tpert}
that meets $\{(0,\hat p,\hat q)\in\Omega:F(p,q)=(0,0)\}$ and cannot be both bounded 
and contained in $\Omega$.

If, in particular, $\Omega=[0,\infty)\times A$, with $A\subseteq C_T(U)$ open, bounded, 
and such that there are no $T$-periodic solutions of \eqref{Tpert} on the boundary 
$\Fr (A)$ of $A$ for $\lambda\in [0,1]$, then equation 
\begin{equation}\label{contin}
 \left\{
\begin{array}{l}
 \dot x=f(x,y)+h(t,x,y),\\
 g(x,y)=0.
\end{array}
\right.
\end{equation}
admits a $T$-periodic solution in $A$.
\end{corollary}

\begin{proof}
By Theorem \ref{rami}, there exists a connected set $\Gamma$ of nontrivial solution 
pairs of \eqref{Tpert} whose closure $\cl\Gamma$ in $[0,\infty)\times C_T(U)$ meets 
the set $\{(0,\hat p,\hat q)\in\Omega:F(p,q)=(0,0)\}$ and is not contained in any 
compact subset of $\Omega$. Let $\Sigma$ be the connected component of the set of all
solution pairs that contains $\cl\Gamma$.

Since $M\subseteq\R^k\times\R^s$ is closed, the metric space $[0,\infty)\times C_T(M)$ 
is complete. Moreover, the Ascoli-Arzel\`a Theorem implies that any bounded set of 
$T$-periodic solutions of \eqref{TeqonM} is totally bounded. Thus, if $\Sigma$ is 
bounded, then it is also compact. If, in addition, $\Sigma$ is contained in $\Omega$ 
then so is $\Gamma\subseteq\Sigma$, which is impossible. This contradiction proves that 
$\Sigma$ cannot be both bounded and contained in $\Omega$.

To prove the last part of the assertion observe that $\Sigma$ is connected and that
$\emptyset\neq\Sigma\cap\Omega\neq\Sigma$. Thus, $\Sigma$ necessarily meets the boundary 
of $\Omega=[0,\infty)\times A$. Since there are no solution pairs of \eqref{Tpert} in 
$[0,1]\times\Fr A$, one has that $\Sigma$ intersects $\{1\}\times A$. 
\end{proof}

\begin{remark}
A practical method of applying Corollary \eqref{corami} is to consider a relatively
compact open subset $V$ of $U$ with the following properties:
\begin{itemize}
  \item the set $F^{-1}(0,0)\cap V$ is compact and  $\deg(F,V)\neq 0$;
  \item there is no $T$-periodic solution of \eqref{Tpert} whose image intersects 
        the boundary of $V$. (This last point might be difficult to verify and is usually 
            proved by the means of a priori bounds.)
\end{itemize}
In this situation, taking $A= C_T(V)$ and $\Omega=[0,\infty)\times A$, we have
$U\cap\Omega=V$ and
\[
 \deg(F,U\cap\Omega)=\deg(F,V)\neq 0.
\]
Hence, Corollary \ref{corami} yields a $T$-periodic solution of \eqref{contin}.
\end{remark}

The next two subsections, that are meant mainly as illustrations of Theorem \ref{rami} 
and of its main consequence Corollary \ref{corami}, are each devoted to a quite different 
application.

\subsection{Example of application to multiplicity results}
This subsection is devoted to some multiplicity results that can be deduced from 
Theorem \ref{rami} and from its Corollary \ref{corami}. Throughout this subsection $f$, 
$g$, $h$, $U$, $T$ and $F$ will be as in Theorem \ref{rami} and, in addition, we 
will assume that $f$ is $C^1$.

In order to obtain multiplicity results, we combine the global approach of
Theorem \ref{rami} with a local analysis of the set of $T$-periodic solutions. 
Let $(p_0,q_0)$ be an isolated zero of $F$. Since $\partial_2g(p_0,q_0)$ is invertible, 
we can locally ``decouple'' \eqref{Tpert}. Namely, by the Implicit Function Theorem, 
there exist neighborhoods $V\subseteq\R^k$ of $p_0$ and $W\subseteq\R^s$ of $q_0$, and 
a function $\gamma:V\to\R^s$ such that $g^{-1}(0)\cap V\times W$ is the graph of $\gamma$. 
Thus, in $V\times W$, equation \eqref{Tpert} can be written as
\begin{subequations}\label{bidec}
\begin{equation}\label{decoup} 
\dot x= f\big(x,\gamma(x)\big)+\lambda h\big(t,x,\gamma(x)\big),
\end{equation}
\begin{equation}
\quad y=\gamma(x).
\end{equation}
\end{subequations}
We will say that $(p_0,q_0)$ is a $T$-resonant zero of $F$ if the following linearization, for 
$\lambda=0$, of \eqref{decoup} at $(p_0,q_0)$:
\begin{equation}\label{linearz}
\dot\xi=\big[\partial_1 f(p_0,q_0)
             -\partial_2f(p_0,q_0)d_{(p_0,q_0)}\gamma\big]\xi 
\end{equation}
admits nonzero $T$-periodic solutions (note that \eqref{linearz} is an ordinary differential 
equation in $\R^k$).

A simple computation shows that $(p_0,q_0)$ is $T$-resonant if and only if the
following linear endomorphism of $\R^k$:
\begin{equation}\label{diffP}
\partial_1 f(p_0,q_0)
             -\partial_2f(p_0,q_0)[\partial_2g(p_0,q_0)]^{-1}\partial_1g(p_0,q_0) 
\end{equation}
has eigenvalues of the form $2n\pi i/T$, where $n\in\N\cup\{0\}$, and $i$ denotes the 
imaginary unit. Also, the generalized Gauss algorithm, as in the proof of Theorem 
\ref{formuladeg}, yields
\begin{multline*}
 \det d_{(p_0,q_0)}F=\det\big(\partial_2g(p_0,q_0)\big)\cdot\\
                      \cdot\det\Big(\partial_1 f(p_0,q_0)
             -\partial_2f(p_0,q_0)[\partial_2g(p_0,q_0)]^{-1}\partial_1g(p_0,q_0)\Big).
\end{multline*}
Thus, if $(p_0,q_0)$ is non-$T$-resonant, then it is a nondegenerate zero of $F$. Hence, 
$\idx\big(F,(p_0,q_0)\big)\neq 0$.

 From Theorem \ref{rami} we get the following lemma:

\begin{lemma} \label{rametto}
Assume that $f$, $g$, $h$, $U$, $T$ and $F$ be as in Theorem \ref{rami}. Assume also that
$f$ is $C^1$ and let $(p_0,q_0)$ be a non-$T$-resonant zero of $F$. Then
\begin{enumerate}
\item the trivial $T$-pair $(0;\hat p_0,\hat q_0)$ is isolated in the set of 
$T$\hbox{-}pairs corresponding to $\lambda =0$;
\item there exists a connected set of nontrivial $T$\hbox{-}pairs of \eqref{Tpert} whose 
closure in $[0,\infty)\times C_T(U)$ contains $(0;\hat p_0,\hat q_0)$ and is noncompact or 
intersects the set 
\[
\big\{(0;\hat p,\hat q)\in [0,\infty)\times C_T(U):F(p,q)=(0,0)\big\}\setminus
  \{(0;\hat p_0,\hat q_0)\}
\]
\end{enumerate}
\end{lemma}
\begin{proof} Let us prove the first part of the assertion. Assume by contradiction that 
there exists a sequence $\{(0;x_n,y_n)\}$, $n=1,2,\ldots$, of $T$-pairs of \eqref{Tpert} 
with $(x_n,y_n)\to (\hat p_0,\hat q_0)$ uniformly. If we put $p_n=x_n(0)$, we clearly have 
$p_n\to p_0$. We claim that this is not possible. 
Let $V$ and $\gamma$ be as in \eqref{bidec}. For any $p'\in V$, denote by $x(\cdot,p')$ 
the maximal solution of the Cauchy problem 
\[
 \left\{
\begin{array}{l}
 \dot x= f\big(x,\gamma(x)\big),\\
  x(0)=p'
\end{array}
\right.
\]
Well known results in the theory of ordinary differential equations imply that there exists
an open neighborhood $W\subseteq V$ of $p_0$ such that the map $P$, that to $p'\in W$ associates $x(T,p')\in\R^k$, is defined. Also, since $p\mapsto f\big(p,\gamma(p)\big)$ is continuous in 
$W\subseteq V$, we know that $P$ is $C^1$ in $W$ and that its differential $d_{p_0}P$ is given 
by \eqref{diffP}. Thus, since $(p_0,q_0)$ is a nondegenerate zero of $F$, the linear operator 
$d_{p_0}P$ is invertible. The claim now follows from the Inverse Function Theorem.

Let us prove the second part of the assertion. Since $(0;\hat p_0,\hat q_0)$ is isolated 
in the set of $T$-pairs corresponding to $\lambda =0$, the set 
\[
\big\{(0;\hat p,\hat q)\in [0,\infty)\times C_T(U):F(p,q)=(0,0)\big\}
\setminus\{(0;\hat p_0,\hat q_0)\}
\]
is closed. Thus, the set 
\begin{multline*}
 \Omega=\Big([0,\infty)\times C_T(U)\Big)\setminus\\
        \Big(\big\{(0;\hat p,\hat q)\in [0,\infty)\times C_T(U):F(p,q)=(0,0)\big\}
\setminus\{(0;\hat p_0,\hat q_0)\}\Big)
\end{multline*}
is open. As in Theorem \ref{rami}, we use the symbol $U\cap\Omega$ as a shorthand notation for 
the set $\{(p,q)\in U:(0,\hat p,\hat q)\in\Omega\}$. Since $(p_0,q_0)$ is non-$T$-resonant and 
since $U\cap\Omega$ is just the singleton $\{(p_0,q_0)\}$, from \eqref{sommaindici} we have
\[
 \deg(F,U\cap\Omega)=\idx\big(F,(p_0,q_0)\big)\neq 0.
\] 
By Theorem \ref{rami} there exists a connected set $\Gamma$ of $T$-pairs that meets
\[
\big\{(0;\hat p,\hat q)\in\Omega:F(p,q)=(0,0)\big\}=\{(0;\hat p_0,\hat q_0)\},
\]
and is such that its closure $\cl{\Gamma}$ in $[0,\infty)\times C_T(U)$ is not contained in 
any compact subset of $\Omega$. Hence $\Gamma$ satisfies the second part of the assertion.
\end{proof}

Let us introduce some notation. Let $Y$ be a metric space and $X$ a subset of 
$[0,\infty )\times Y$. Given $\lambda \geq 0$, we denote by $X _{\lambda}$ the slice 
$\big\{y\in Y:(\lambda ,y)\in X \big\}$. Recall the following notion from \cite{FPS00}:
We say that $A\subseteq X_0$ is an {\it ejecting set} (for $X$) if it is relatively open 
in $X_0$ and there exists a connected subset of $X$ which meets $A$ and is not contained 
in $X_0$. For example,  any non-$T$-resonant point of \eqref{Tpert} is an ejecting set 
(or, rather, \textit{ejecting point}). In fact, as a consequence of Lemma \ref{rametto}, 
if $X$ denotes the set of $T$-pairs of \eqref{Tpert} and $Y=C_T(U)$, any non-$T$-resonant 
zero of $F$ turns out to be an isolated point of $X_0$ which is ejecting.

Let us recall the following abstract result from \cite{FPS00}:

\begin{theorem}
\label{conn2} Let $Y$ be a metric space and let $X$ be a locally compact subset of 
$[0,\infty )\times Y$. Assume that $X_0$ contains $r+1$ pairwise disjoint ejecting 
subsets, $r$ of which are compact. Then there exists $\lambda _{*}>0$ such that the
cardinality of $X_\lambda$ is at least $r+1$ for any $\lambda \in [0,\lambda _{*})$.
\end{theorem}

We are now in a position to state and prove the following multiplicity result:

\begin{proposition}\label{mults}
 Let $f$, $g$, $h$, $U$, $M$, $T$, $F$ and $\Omega$ be as in Theorem \ref{rami}. 
Assume also that $f$ is $C^1$ and that $M=g^{-1}(0)$ is closed in $\R^k\times\R^s$. 
Let $(p_1,q_1)$,\ldots, $(p_{r},q_{r})$ be non-$T$-resonant zeros of $F$ such that 
\[
\deg(F,U)\neq\sum_{j=1}^{r} \idx\big(F,(p_j,q_j)\big).
\]
Suppose that \eqref{NOpert} does not admit an unbounded connected set of $T$-periodic
solutions in $C_T(U)$. Then, there are at least $r+1$ different $T$-periodic solution 
of \eqref{Tpert} when $\lambda>0$ is sufficiently small.
\end{proposition}

The assumption on the nonexistence of an unbounded connected set of $T$-periodic
solutions (in $C_T(U)$) of the unperturbed equation \eqref{NOpert} is often the most 
difficult to verify and usually shown to hold with the help of a priori bounds. 

\begin{proof}[Proof of Proposition \ref{mults}]
 Let
\[
 \Omega=\Big([0,\infty)\times C_T(U)\Big)\setminus 
         \bigcup_{j=1}^{r}\{(0;\hat p_j,\hat q_j)\}.
\]
By the additivity property of the degree and formula \eqref{sommaindici}
\[
 \deg(F,U\cap\Omega)=\deg(F,U)-\sum_{j=1}^{r} \idx\big(F,(p_j,q_j)\big)\neq 0,
\]
where, as in Theorem \ref{rami}, we use the notation $U\cap\Omega=\{(p,q)\in U:
(0,\hat p,\hat q)\in\Omega\}$.
Let $X$ be the set of all $T$-pairs of \eqref{Tpert}. By Corollary \ref{corami}, 
$M=g^{-1}(0)$ being closed in $\R^k\times\R^s$, there exists a connected component 
$\Gamma$ of $X$ that cannot be both bounded and contained in $\Omega$. Since
by assumption \eqref{NOpert} does not admit an unbounded connected set of 
$T$-periodic solutions in $C_T(U)$, it is not difficult to show that the set 
\[
X_0\setminus\bigcup_{j=1}^{r}\{(0;\hat p_j,\hat q_j)\}
\]
is ejecting. The assertion now follows from Lemma \ref{rametto} and Theorem \ref{conn2}.
\end{proof}

As an illustration of Proposition \ref{mults} we consider the following elementary example
even though, in that case, the situation is sufficiently simple to be treatable without the 
help of our multiplicity result.

\begin{example}
 Consider the following DAE in $U=\R\times\R$:
\begin{equation}\label{EX.mults}
 \left\{
\begin{array}{l}
 \dot x = y^2-xy+\lambda h(t,x,y),\\
 y-x^2=0,
\end{array}
\right.
\end{equation}
where $h:\R\times\R\times\R\to\R$ is any continuous function $T$-periodic in the first
variable. 

One immediately sees that $F(p,q)=(q^2-pq,q-p^2)$ has only the two zeros $(0,0)$ and $(1,1)$ 
of which the former is $T$-resonant, whereas the latter is not so. Also, it not difficult to 
see that for $\lambda=0$ the only possible periodic solutions of \eqref{EX.mults} correspond 
to the zeros of $F$. Thus, for $\lambda=0$, equation \eqref{EX.mults} does not admit an 
unbounded connected set of $T$-periodic solutions. 

By inspection, we see that the homotopy $H:U\times[0,1]\to\R\times\R$, given by 
$H(p,q;\lambda)=(q^2-pq,q-p^2-\lambda)$, is admissible. Since $H(p,q;1)\neq 0$ for
any $(p,q)\in U$, we have $\deg\big(H(\cdot,\cdot;1),U\big)=0$. Thus, by Homotopy Invariance, $\deg(F,U)=0$.
 
Since a non-$T$-resonant zero of $F$ is nondegenerate, we have $\idx\big(F,(1,1)\big)\neq 0$. 
Hence, by Proposition \ref{mults}, for sufficiently small $\lambda>0$ there are at least two 
$T$-periodic solutions of \eqref{EX.mults}
\end{example}

\subsection{Example of application to a class of implicit differential equations}

In this subsection we will describe an application to periodic perturbations of 
ordinary differential equations of a particular implicit form. What follows is
mostly intended as an illustration of Theorem \ref{rami} and of its Corollary \ref{corami}. 
For this reason we do not seek generality but confine ourselves to a fairly simple 
situation. Namely, we consider the following equation:
\begin{equation}\label{impl}
\varphi\big(x,\dot x+\lambda h(t,x)\big)=0,
\end{equation}
where $\varphi:\R^k\times\R^k\to\R^k$ is $C^\infty$ with the property that 
$\partial_2\varphi(p,q)$ is invertible for all $(p,q)\in\R^k\times\R^k$; and 
$h:\R\times\R^k\to\R^k$ is continuous and $T$-periodic in the first variable, with given 
$T>0$.

We will also need the following \emph{``no blow up''} assumption on $\varphi$. Namely, 
we suppose that $\varphi$ is such that:
\begin{equation}\label{noblow}
\parbox{0.58\linewidth}{The set $\big\{ q\in\R^k:\varphi(p,q)=0,\, p\in K\big\}$
        is bounded for any bounded $K\subseteq\R^k$.}
\end{equation}

Before we proceed we need a technical result on the degree of a special class of vector 
fields. Its proof is inspired by the one of Theorem 6.1 in \cite{De80} regarding the Brouwer 
degree. 

\begin{lemma}\label{priduz}
Let $V\subseteq\R^k\times\R^k$ be open, $\omega:\R^k\times\R^k\to\R^k$ be continuous and such 
that $[\omega(\cdot,0)]^{-1}(0)\cap V$ is compact. Define $v:\R^k\times\R^k\to\R^k\times\R^k$ 
by $v(p,q)=\big(q,\omega(p,q)\big)$. Then, $(v,V)$ is admissible and
\[
\deg(v,V)=-\deg\big(\omega(\cdot,0),V_0\big),
\]
where $V_0=\{p\in\R^k:(p,0)\in V\}$.
\end{lemma} 

\begin{proof}
By the Excision Property and the compactness of $[\omega(\cdot,0)]^{-1}(0)\cap V$, taking a 
smaller $V$ if necessary, we can assume that $V$ is bounded and such that $\omega(p,0)\neq 0$ 
for all $(p,0)$ in the boundary $\Fr(V)$ of $V$. 

Observe that the homotopy $H:\R^k\times\R^k\times[0,1]\to\R^k\times\R^k$ given by 
$H(p,q,\lambda)=\big(q,\omega(p,\lambda q)\big)$ is admissible in $V$, define 
$G(p,q)= H(p,q,0)$. By the Homotopy Invariance Property it is sufficient to show that 
$\deg(G,V)=-\deg\big(\omega(\cdot,0),V_0\big)$.

By known approximation results (see e.g.\ \cite{GP} or \cite{H}) there exists a $C^1$  
map $\eta:\R^k\to\R^k$, such that all its zeros contained in $V_0$ are nondegenerate, and 
with the property that
\[
 \max_{p\in\Fr(V_0)}|\eta(p)-\omega(p,0)|_k<\min_{p\in\Fr(V_0)}|\omega(p,0)|_k.
\]
As in Remark \ref{sper} we have that 
\begin{equation}\label{gradoeta}
 \deg(\eta,V_0)=\deg\big(\omega(\cdot,0),V_0\big).
\end{equation}
Define $Q:\R^k\times\R^k\to\R^k\times\R^k$ by $Q(p,q)=\big(q,\eta(p)\big)$. By Remark 
\ref{sper} again, 
\begin{equation}\label{gradoQ}
 \deg(G,V)=\deg(Q,V),
\end{equation}
since
\begin{align*}
\max_{(p,q)\in\Fr(V)}|Q(p,q)-G(p,q)|_{2k}=&\max_{p\in\Fr(V_0)}|\eta(p)-\omega(p,0)|_k\\
 <&\min_{p\in\Fr(V_0)}|\omega(p,0)|_k\leq\min_{(p,q)\in\Fr(V)}|G(p,q)|_{2k}.
\end{align*}

Observe also that, since all the zeros of $\eta$ in $V_0$ are nondegenerate, so are those 
of $Q$ in $V$. In fact, $Q^{-1}(0,0)=\{(p,0)\in\R^k\times\R^k:\eta(p)=0\}$ and, writing the 
differential $d_{(p,q)}Q$ of $Q$ at any $(p,q)\in Q^{-1}(0,0)$ in block matrix form, we get 
\[
 \det d_{(p,q)}Q=\det\begin{pmatrix}
0 & \mathrm{Id}_{\R^k} \\
d_p\eta & 0
\end{pmatrix}=-\det d_p\eta\,,
\]
where $\mathrm{Id}_{\R^k}$ denotes the identity on $\R^k$ and $d_p\eta$ is the differential 
of $\eta$ at $p$. Thus, taking \eqref{gradoQ}, \eqref{sommasegni} and \eqref{gradoeta} into 
account, we get
\begin{multline*}
\deg(G,V)=\deg(Q,V)=
\sum_{(p,q)\in Q^{-1}(0,0)\cap V}\sign\det d_{(p,q)}Q=\\
=-\sum_{p\in\eta^{-1}(0)\cap V_0}\sign\det d_p\eta=-\deg\big(\eta,V_0\big)=
-\deg\big(\omega(\cdot,0),V_0\big),
\end{multline*}
that implies the assertion.
\end{proof}

Recall that, given $p\in\R^k$, we denote by $\hat p$ the function in $C_T(\R^k)$ that is 
constantly equal to $p$. Given an open set $W\subseteq [0,\infty)\times C_T(\R^k)$, it is 
convenient to denote by $\R^k\cap W$ the open subset of $\R^k$ given by $\big\{ p\in\R^k: 
(0,\hat p)\in W\big\}$.

\begin{proposition}
Let $h$ and $\varphi$ be as above, and let $W\subseteq [0,\infty)\times C_T(\R^k)$ be 
open and such that  $\deg\big(\varphi(\cdot,0),\R^k\cap W\big)$ is well defined and nonzero. 
Then the subspace of $[0,\infty)\times C_T(\R^k)$ consisting of all pairs $(\lambda,x)$, 
with $x$ a (clearly $T$-periodic) solution of \eqref{impl}, contains a connected component 
$\Xi$ that intersects the set
\[
\big\{(0,\hat p)\in W:\varphi(p,0)=0\big\}
\]
and is unbounded or meets the boundary of $W$. 
\end{proposition}

\begin{proof}
Clearly, \eqref{impl} is equivalent to the following DAE:
\begin{equation}\label{equimpl}
\left\{
\begin{array}{l}
\dot x = y-\lambda h(t,x),\\
\varphi(x,y)=0.
\end{array}\right.
\end{equation}
Let $F:\R^k\times\R^k\to\R^k\times\R^k$ be given by $F(p,q)=\big(q,\varphi(p,q)\big)$.
Denote by $\Omega$ the set 
\[
\big\{(\lambda; x,y)\in[0,\infty)\times C_T(\R^k\times\R^k)
                                   :(\lambda,x)\in W,\,y\in C_T(\R^k)\big\},
\]
and by $V$ the set $\{(p,q)\in\R^k\times\R^k:(0;\hat p,\hat q)\in\Omega\}$. Let
$V_0=\{p\in\R^k:(p,0)\in V\}$. Since $V_0=\R^k\cap W$, Lemma \ref{priduz} implies that
\begin{equation}\label{gradino}
 \deg\big(F,V\big)=\deg\big(\varphi(\cdot,0),\R^k\cap W\big)\neq 0.
\end{equation}
Using the notation $(\R^k\times\R^k)\cap\Omega=\{(p,q)\in\R^k\times\R^k:(0;\hat p,\hat 
q)\in\Omega\}$ as in Theorem \ref{rami}, we write $V=(\R^k\times\R^k)\cap\Omega$.
By \eqref{gradino}, $\deg\big(F,(\R^k\times\R^k)\cap\Omega\big)=\deg\big(F,V\big)\neq 0$.
Thus, Corollary \ref{corami} yields the existence of a connected component $\Gamma$ of the set 
of solutions pairs of \eqref{equimpl} that meets 
\[
\{(0;\hat p,\hat q)\in\Omega:F(p,q)=0\}=\{(0;\hat p,\hat 0):\varphi(p,0)=0\}
\]
(here $\hat 0\in C_T(\R^k)$ denotes the map $\hat 0(t)\equiv 0\in\R^k$) and is unbounded 
or intersects the boundary of $\Omega$. Let $\Xi\subseteq [0,\infty)\times C_T(\R^k)$ be 
the connected set defined by
\[
\Xi=\Big\{ (\lambda,x)\in[0,\infty)\times C_T(\R^k): \exists y\in C_T(\R^k)\text{ s.t.\ } 
               \big(\lambda;x, y\big)\in\Gamma\Big\}.
\]
Notice that if $\Xi$ is contained in $W$, then $\Gamma\subseteq\Omega$. Also, using  
assumption \eqref{noblow}, it is not difficult to prove that if $\Xi$ is bounded 
then so is $\Gamma$. Hence, $\Xi$ cannot be both bounded and contained in $W$. The 
assertion follows from the connectedness of $\Xi$.
\end{proof}


\begin{thebibliography}{9}
\bibitem{De80} K.\ Deimling, \textit{Nonlinear Functional Analysis}, Springer Verlag, 
New York, 1980.

\bibitem{FPS00} M.\ Furi, M.\ P.\ Pera, M.\ Spadini, \textit{Forced oscillations on
manifolds and multiplicity results for periodically perturbed autonomous systems},
J.\ Computational and App.\ Mathematics 113 (2000), 241--254.

\bibitem{FPS05} M.\ Furi, M.\ P.\ Pera, M.\ Spadini, \textit{The fixed 
point index of the Poincar\'e operator on differentiable manifolds}, Handbook of topological 
fixed point theory, Brown R.\ F., Furi M., G\'orniewicz L., Jiang B.\ (Eds.), Springer, 2005.

\bibitem{FS98} M.\ Furi and M.\ Spadini, \textit{On the set of harmonic solutions of 
periodically perturbed autonomous differential equations on manifolds}, Nonlinear Analysis TMA 
Vol. 29 (1997) n. 8.

\bibitem{G} F.\ R.\ Gantmacher, \textit{Th\'eorie des matrices (tome 1)},
Collection Univ. de Math\'ematiques, Dunod, Paris 1966.

\bibitem{GP}  V.\ Guillemin and A.\ Pollack, \textit{Differential Topology,}
Prentice-Hall Inc., Englewood Cliffs, New Jersey, 1974.

\bibitem{H} M.\ W.\ Hirsch, \textit{Differential Topology},  Graduate Texts 
in Math. Vol. 33, Springer Verlag, Berlin, 1976.

\bibitem{KM} P.\ Kunkel, V.\ Mehrmann, \textit{Differential Algebraic Equations, analysis and
numerical solutions}, EMS textbooks in Math., European Mathematical Society, Z\"urich 2006.

\bibitem{Mi} J.\ W.\ Milnor, \textit{Topology from the differentiable 
viewpoint}, Univ. press of Virginia, Charlottesville, 1965.

\end{thebibliography}
\end{document}